\documentclass[11pt]{amsart}
\usepackage{amsmath,amssymb,amsthm,mathtools}
\usepackage{microtype}
\usepackage[hidelinks]{hyperref}
\usepackage[margin=1.08in]{geometry}
\usepackage{booktabs,array,tabularx}
\usepackage[T1]{fontenc}
\usepackage{lmodern}
\numberwithin{equation}{section}

\newtheorem{theorem}{Theorem}[section]
\newtheorem{proposition}[theorem]{Proposition}
\newtheorem{corollary}[theorem]{Corollary}
\newtheorem{lemma}[theorem]{Lemma}
\theoremstyle{definition}
\newtheorem{definition}[theorem]{Definition}
\theoremstyle{remark}
\newtheorem{remark}[theorem]{Remark}

\newcommand{\CH}{\operatorname{CH}}
\newcommand{\Pic}{\operatorname{Pic}}
\newcommand{\NS}{\operatorname{NS}}
\newcommand{\Spec}{\operatorname{Spec}}

\newcommand{\im}{\operatorname{im}}
\newcommand{\Tor}{\operatorname{Tor}}

\newcommand{\mor}{\mathrm{mor}}
\newcommand{\sing}{\mathrm{sing}}
\newcommand{\MZ}{\mathrm{M}\mathbb Z}
\newcommand{\HZ}{H\mathbb Z}
\newcommand{\HR}{HR}
\newcommand{\dsmash}{\mathbin{\wedge^{\mathbf L}}}
\newcommand{\boxt}{\mathbin{\boxtimes}}
\newcommand{\C}{\mathbb C}
\newcommand{\Z}{\mathbb Z}

\newcommand{\Dsc}{D_{\mathrm{sc}}}

\title[Secondary Chow Invariants]{Secondary Chow Invariants beyond Lawson--Massey Products}
\author{Wenchuan Hu}
\address{School of Mathematics, Sichuan University, Chengdu 610064, China}
\email{huwenchuan@gmail.com}
\email{wenchuan@scu.edu.cn}
\date{\today}
\subjclass[2020]{Primary 14C25; Secondary 14F42, 14F43, 55S30, 14H52}
\keywords{secondary invariant, higher Chow group, Massey product, motivic cohomology, Weil pairing, morphic cohomology, Lawson homology}

\begin{document}
\begin{abstract}
We construct secondary Chow invariants that cannot be recovered from Lawson–Massey realization. For every smooth complex projective curve \(C\) of positive genus and every \(n\ge2\), explicit cubical defining systems yield \(n\) distinct Massey values in \(\mathrm{CH}^{3}(C\times\mathbb P^{2},1)\), with zero full indeterminacy, forming a cyclic group of order \(n\). All corresponding triples have identical morphic classes and zero morphic, Lawson, and singular Massey brackets, despite nonzero target groups. The values are detected by the Weil pairing. We further construct secondary quotient classes of every finite order on smooth projective varieties with ample canonical bundle in every dimension at least two, and prove that generically finite pullbacks of coprime degree preserve their exact orders. These constructions obstruct integral and prime-local motivic formality, although the corresponding brackets become \(\{0\}\) after reduction modulo every integer \(m\ge2\). All specified values are torsion and vanish after rationalization.
\end{abstract}
\maketitle

\section{Introduction}\label{sec:introduction}
The purpose of this paper is to identify secondary information carried by Chow classes that is lost even by the Lawson--Massey operation.  The distinction is between two realization steps.  A secondary multiplicative structure has been constructed on Lawson homology and used to exhibit nonzero Lawson--Massey quotient classes that are killed by the Lawson cycle map to singular homology \cite{HuLawson}.  Here we work before semi-topological realization: we exhibit nonzero motivic Massey quotient classes whose corresponding morphic, and hence Lawson, brackets are already zero.  The central assertion is not merely that the source of a realization map contains torsion.  We construct actual defining systems, compute the entire indeterminacy, and show that distinct secondary values have the same realized ordered triple and the same Lawson--Massey bracket.

\subsection{The secondary operation and its invariant}
For a smooth complex quasi-projective variety $X$, let $\Pic(X)$ denote the group of line bundles up to isomorphism, and let $\CH^p(X,j)$ denote Bloch's higher Chow group of codimension $p$ and homological degree $j$; $\CH^p(X)=\CH^p(X,0)$ is the ordinary Chow group of cycles modulo rational equivalence.  We use the standard comparison
\begin{equation}\label{eq:motivic-higher Chow}
 H^m_M(X,\Z(p))\cong\CH^p(X,2p-m),
 \qquad H^{2p}_M(X,\Z(p))=\CH^p(X),
\end{equation}
where $H^m_M$ denotes motivic cohomology and $\Z(p)$ is the weight-$p$ motivic complex \cite{Bloch1986,Voevodsky2002}; see also \cite[Theorem~19.1]{MVW}.  For $a\in\CH^p(X)$, $b\in\CH^q(X)$, and $c\in\CH^r(X)$ with $ab=bc=0$, the motivic triple bracket is a nonempty coset in $\CH^s(X,1)$, where $s=p+q+r$, with full indeterminacy
\begin{equation}\label{eq:intro-indet}
 I_M(a,b,c)=a\,\CH^{q+r}(X,1)+\CH^{p+q}(X,1)\,c.
\end{equation}
Its well-defined class in that quotient will be denoted
\begin{equation}\label{eq:secondary-invariant}
 \mathfrak m_{\CH,X}(a,b,c):=[\langle a,b,c\rangle_{\CH}]
 \in\CH^s(X,1)/I_M(a,b,c).
\end{equation}
This is an invariant of the variety \emph{together with the ordered triple}, in a quotient depending on that triple.  It is not a numerical invariant of $X$ alone.  The bracket is the partially defined, set-valued operation; its quotient class is the secondary invariant.  Their compatible collection is a secondary multiplicative structure on the motivic enhancement, not a new multiplication or a claimed canonical $A_\infty$-structure on the ordinary Chow ring.

There are multiplicative realization maps
\begin{equation}\label{eq:realizations-intro}
 H^m_M(X,\Z(p))\xrightarrow{\rho}L^pH^m(X;\Z)
 \xrightarrow{\mathrm{cl}}H^m(X^{\mathrm{an}};\Z).
\end{equation}
Here $L^pH^m$ is morphic cohomology, $X^{\mathrm{an}}$ is the associated complex analytic space, and the last group is singular cohomology.  Heller supplies the multiplicative semi-topological comparison \cite[Theorem~4.11(2), Proposition~5.6]{Heller}.  Constant Betti Tate twists are identified using the usual generators $(2\pi i)^p$.  For a smooth projective $d$-fold, Friedlander--Lawson duality identifies the middle group with $L_{d-p}H_{2d-m}(X)$ for $0\le p\le d$ and $m\le2p$ \cite{FLDuality}.  Under this identification the second arrow is the Lawson cycle map
\[
 \Phi_{d-p,2d-m}:L_{d-p}H_{2d-m}(X)\longrightarrow H_{2d-m}(X^{\mathrm{an}};\Z)
\]
followed by inverse Poincar\'e duality.  Thus singular cohomology and singular homology play distinct, dual roles.

Naturality induces maps between the quotients by their respective full indeterminacies:
\begin{equation}\label{eq:quotient-chain-intro}
 Q_M(a,b,c)\xrightarrow{\bar\rho}
 Q_{\mor}(\rho a,\rho b,\rho c)\xrightarrow{\overline{\mathrm{cl}}}
 Q_{\sing}(\mathrm{cl}\rho a,\mathrm{cl}\rho b,\mathrm{cl}\rho c).
\end{equation}
These quotients are defined explicitly in Section~\ref{sec:lawson-comparison}.  Our examples give $\mathfrak m_{\CH,X}(a,b,c)\ne0$ in $\ker\bar\rho$.  By contrast, the motivic classes in \cite[Section~13]{HuLawson} have nonzero images in $Q_{\mor}$ that are killed by $\overline{\mathrm{cl}}$.  Accordingly, ``beyond Lawson--Massey'' means a proved failure of recovery from the \emph{corresponding realized triple and its bracket}.  It does not mean that the classical motivic operation is new or that the information escapes every conceivable invariant of the variety.

\subsection{Background and attribution}
Massey introduced higher cohomology operations on singular cochains; Kraines and May developed their systematic and matric formulations \cite{Massey1958,Kraines1966,May1969}.  Such operations retain null-homotopies of vanishing products and can obstruct formality.  On compact K\"ahler manifolds, the Deligne--Griffiths--Morgan--Sullivan theorem forces every defined rational singular Massey product to contain zero \cite{DGMS}.

Secondary operations in refined cohomology have  precedents.  Deninger and Schwarzhaupt study Deligne and Deligne--Beilinson Massey products \cite{Deninger1995,Schwarzhaupt2002}; Grady--Sati treat differential-cohomological versions \cite{GradySati}; and Guillou uses motivic Massey products in the theory of the punctured projective line \cite{Guillou}.  For the geometry used here, Bloch's Chow biextensions \cite{Bloch1989} and Gorchinskiy's treatment through pairings of higher Chow complexes and Weil pairings \cite[Example~3.10, Section~4.1, and Corollary~4.36]{GorchinskiyBiext} are especially close predecessors; see also the $K$-cohomological calculation in \cite[Propositions~4.25 and~4.27]{Gorchinskiy}.

The feature established here is the separation of secondary quotient classes under realization: explicit integral cubical representatives give nonzero classes modulo the \emph{full indeterminacy} that disappear before the Lawson stage.  The fixed realization family below makes this separation concrete.  The construction also yields geometric and formality applications without using the existence results of \cite{HuLawson}.

\subsection{Main results}
We first record the classical construction in the grading needed here; the proof is included only to fix signs, degrees, and the full quotient.
\begin{theorem}[Construction and full indeterminacy]\label{thm:intro-construction}
Let $X$ be smooth and complex quasi-projective, and let $p,q,r\ge0$.  If $a\in\CH^p(X)$, $b\in\CH^q(X)$, and $c\in\CH^r(X)$ satisfy $ab=bc=0$, their motivic triple bracket is a nonempty coset
\[
 \langle a,b,c\rangle_{\CH}\subset\CH^{p+q+r}(X,1)
\]
modulo exactly \eqref{eq:intro-indet}.  It defines \eqref{eq:secondary-invariant}, independently of the cocycle representatives and of a multiplicatively equivalent model.  An $N$-fold bracket of ordinary Chow classes, for $N\ge3$ and a full defining system, takes values in $\CH^{\sum p_i}(X,N-2)$.
\end{theorem}

For a smooth projective curve $C$, $J(C)=\Pic^0(C)$ is its Jacobian, the connected group of degree-zero line-bundle classes, and $J(C)[n]=\{x:nx=0\}$ is its $n$-torsion subgroup.  Write $\mu_n\subset\C^\times$ for the $n$th roots of unity and $e_n:J(C)[n]\times J(C)[n]\to\mu_n$ for the Weil pairing of its canonical principal polarization.  The divisor convention is fixed in Section~\ref{sec:curve}, and $\operatorname{ord}(v)$ denotes the additive order of a group element.
\begin{theorem}[Exact torsion on every positive genus curve]\label{thm:intro-curves}
Let $C$ be a smooth connected complex projective curve with $g(C)\ge1$, and let $n\ge2$.  There are classes $\alpha,\gamma\in J(C)[n]$ such that
\[
 \langle\alpha,n\cdot1,\gamma\rangle_{\CH}=\{\theta\}
 \subset\CH^2(C,1),\qquad \operatorname{ord}(\theta)=n.
\]
The full indeterminacy is zero.  For disjoint divisors $D,D'$ representing these classes and functions $f,g$ with $\operatorname{div}(f)=nD$, $\operatorname{div}(g)=nD'$, a representative is
\[
 \Theta=D\cdot\Gamma_g-\Gamma_f\cdot D',\qquad
 \partial\Gamma_{(f,g)}=n\Theta,\qquad
 (\pi_C)_*[\Theta]=e_n(\alpha,\gamma)^{-1}.
\]
Here $\pi_C:C\to\Spec\C$, and $\Gamma_f$, $\Gamma_g$, $\Gamma_{(f,g)}$ are the admissible graph chains defined in Section~\ref{sec:curve}.
\end{theorem}

\begin{theorem}[Distinct secondary values over a fixed realized triple]\label{thm:intro-loss-family}
Let $C$ and $n$ be as above.  On $Y=C\times\mathbb P^2$, let $p:Y\to C$ and $h=c_1(\mathcal O_{\mathbb P^2}(1))$, pulled back to $Y$.  One can choose
\[
 a=p^*\alpha,\qquad b=nh,\qquad c_k=k\,p^*\gamma\quad(0\le k<n)
\]
so that all adjacent products vanish, all full motivic indeterminacies are zero, and
\[
 \langle a,b,c_k\rangle_{\CH}=\{k\eta\}\subset\CH^3(Y,1),
 \qquad\operatorname{ord}(\eta)=n.
\]
These $n$ values are distinct.  Nevertheless every triple has the same morphic realization $(0,n\rho h,0)$, and each corresponding morphic, Lawson, and singular bracket is $\{0\}$.  The target groups are nonzero:
\[
 L^3H^5(Y)\cong L_0H_1(Y)\cong\Z^{2g(C)},
 \qquad H^5(Y^{\mathrm{an}};\Z)\cong\Z^{2g(C)}.
\]
In particular, the secondary Chow invariant does not factor through the realized ordered triple together with its Lawson--Massey bracket.
\end{theorem}

The assertion for a fixed realized triple is proved in Theorem~\ref{thm:fixed-realized-triple}.  It compares values in \emph{one} higher Chow group with zero indeterminacy, not elements of unrelated quotients.

\begin{theorem}[Examples with ample canonical bundle]\label{thm:intro-ample}
For every $d\ge2$ and $n\ge2$ there is a smooth connected complex projective $d$-fold $Z$ with ample canonical line bundle $K_Z=\omega_Z=\det\Omega_Z^1$ and three divisor classes whose motivic Massey class modulo its full indeterminacy has exact order $n$.  A specified value $\eta_Z\in\CH^3(Z,1)$ has exact order $n$, while the entire realized morphic and singular brackets are $\{0\}$.  One may take $Z$ to be a product of $d$ curves of genus at least two.  If $d\ge3$, the morphic target $L^3H^5(Z)$ is nonzero and lies in the ordinary Lawson range under duality.
\end{theorem}
This is proved by a Projection formula homomorphism in Theorem~\ref{thm:ample-products}; its indeterminacy is \emph{annihilated by this homomorphism}, not asserted to vanish.  For each fixed product of positive genus curves, varying the torsion order gives obstructions to integral formality and to formality over every $\Z_{(\ell)}$.  They do not imply rational nonformality or nonformality of morphic cohomology.

Two further consequences refine this comparison. Proposition~\ref{prop:coprime-transfer} proves that a generically finite pullback whose degree is prime to $n$ preserves the exact order of the secondary quotient class; Corollary~\ref{cor:cyclic-cover} applies this to smooth cyclic covers with ample canonical bundle. In a different direction, Proposition~\ref{prop:finite-zero} proves that, for every integer $m\ge2$, the reduced motivic brackets of the curve based constructions are the singleton $\{0\}$. The latter assertion concerns the entire derived-coefficient bracket, not only the reduction of one selected integral value.

\subsection{Conventions and organization}
Throughout, a smooth variety is assumed equidimensional, and all coefficient rings are commutative and unital. For a connected smooth variety $X$, $1_X\in\CH^0(X)$ denotes its fundamental class, which is the multiplicative unit. For an abelian group $G$, $G[n]=\{x:nx=0\}$ and $G/n=G/nG$.  A bar on a class denotes its image in the specified quotient; a bar on a homomorphism denotes its induced map on quotients.  The ring $\Z_{(\ell)}$ is obtained by inverting the integers prime to $\ell$.  All unmarked groups have integral coefficients, and unmarked tensor products are over $\Z$.  Semicolon coefficients mean derived coefficients, not a tensor product of already computed groups; Section~\ref{sec:coefficients} preserves Lawson's original homotopy-group definition and specifies this convention.  A bracket is nonzero when it does not contain zero.  We write $\Pic(X)$ for line bundles up to isomorphism, $\Pic^0(X)$ for the identity component, and $\NS(X)=\Pic(X)/\Pic^0(X)$ for a smooth projective variety.  The symbols $\rho$ and $\mathrm{cl}$ always refer to the two cohomological maps in \eqref{eq:realizations-intro}; $\Phi$ refers to the Lawson-to-singular-homology map.

Sections~2--3 specify the enhancement and operation.  Sections~4--5 give the curve calculation and an explicit elliptic example.  Section~6 proves the geometric applications, and Section~7 proves the separation from Lawson--Massey realization.  Section~8 distinguishes scalar decomposability from Massey indeterminacy and proves complete vanishing of the reduced brackets for every finite modulus.  Section~9 compares the result with \cite{HuLawson}; the final sections discuss formality and the scope of the conclusions.

\section{Multiplicative motivic models}
\subsection{The multiplicative enhancement}
Let
\[
 M_X=\bigoplus_{p\ge0}M_X^\bullet(p)
\]
be a weight-graded associative differential graded algebra representing the multiplicative motivic cohomology of $X$, so that
\begin{equation}\label{eq:dga-model}
 H^m(M_X(p))\cong H^m_M(X,\mathbb Z(p)),\qquad
 M_X^i(p)\otimes M_X^j(q)\longrightarrow M_X^{i+j}(p+q).
\end{equation}
The phrase ``representing the multiplicative motivic cohomology'' includes the multiplication and its coherent homotopies; an arbitrary differential graded algebra with isomorphic graded cohomology would not suffice.

We recall why an associative integral model is available.  Write $\mathrm{SH}(\mathbb C)$ for the stable motivic homotopy category over $\mathbb C$, and use its standard stable enhancement when forming derived mapping spectra.  A commutative motivic Eilenberg--Mac Lane spectrum $\MZ$ representing integral motivic cohomology is available over this base; see, for example, Spitzweck \cite{Spitzweck}.  Put
\[
 E_X(p)=\mathrm{Map}_{\mathrm{SH}(\mathbb C)}(\Sigma^\infty X_+,\Sigma^{0,p}\MZ)
\]
for the derived ordinary mapping spectrum.  Here $X_+$ is $X$ with a disjoint basepoint, $\Sigma^\infty X_+$ is its motivic suspension spectrum, and $\Sigma^{u,v}$ uses the grading normalized by $[\Sigma^\infty X_+,\Sigma^{m,p}\MZ]=H^m_M(X,\Z(p))$.  Thus $\pi_{-m}E_X(p)=H^m_M(X,\mathbb Z(p))$.  The diagonal of $X$ and multiplication on $\MZ$ give the weight-graded product.  The weight-zero coefficient mapping spectrum
\[
 K_{\mathrm{coeff}}:=\mathrm{Map}_{\mathrm{SH}(\mathbb C)}(\mathbf 1,\MZ)
\]
where $\mathbf1$ is the motivic sphere spectrum, has $\pi_0K_{\mathrm{coeff}}\cong\mathbb Z$ and vanishing other homotopy groups by weight-zero motivic cohomology; compare Mazza--Voevodsky--Weibel \cite[Corollary~4.2]{MVW}.  Hence, as a ring spectrum, it is equivalent to the classical Eilenberg--Mac Lane spectrum $\HZ$.  Consequently the $E_X(p)$ form a weight-graded $\HZ$-linear multiplicative object.  Shipley's Quillen equivalence between associative $\HZ$-algebra spectra and differential graded $\mathbb Z$-algebras supplies an associative model as in \eqref{eq:dga-model}; see \cite[Theorem~1.1]{Shipley2007}.  The weight grading is retained in the category of graded objects, whose tensor product in weight $p$ is $\bigoplus_{i+j=p}A(i)\otimes B(j)$. The monoidal comparison underlying \cite[Theorem~2.5]{Shipley2007} gives the corresponding comparison for associative algebras in these graded categories; the cross-weight products are retained. Naturality is taken in the fixed multiplicative theory, and only a chosen representative is used for a cochain calculation. Neither an arbitrary ring spectrum nor an arbitrary complex with the same cohomology is being substituted for this enhancement. No strictly commutative integral differential graded model is asserted.

\begin{remark}
For the explicit geometric calculations one can instead work in Bloch's cubical cycle complex after moving cycles into good position.  The spectrum-level discussion is used to make model-independence and realization naturality transparent; established comparison theorems are cited rather than reproved.
\end{remark}

\subsection{Cubical representatives and moving}
We use the cube $\square=\mathbb P^1\setminus\{1\}$, with faces obtained by setting coordinates equal to $0$ or $\infty$.  Let $z^p(X,\bullet)$ denote the cubical higher Chow cycle complex modulo degenerate cycles: in degree $j$ it is generated by admissible codimension-$p$ subvarieties of $X\times\square^j$, with cycles pulled back along coordinate projections factored out.  We impose no additional Moore-normalization condition on the graph representatives below. The integral comparison with the simplicial higher Chow complex is \cite[Theorem~4.7]{Levine1994}; the replacement of the usual affine cube by $\mathbb P^1\setminus\{1\}$ is a change of coordinate preserving the two distinguished faces.  The differential is
\[
 \partial=\sum_{i=1}^m(-1)^{i-1}(\partial_i^0-\partial_i^\infty)
 \quad\text{on }z^p(X,m).
\]
Admissible cycles meet every face properly.  Exterior products on separate variety factors are written $\boxt$ and are available directly; internal products require the appropriate proper-intersection condition.  Bloch's moving lemma \cite{Bloch1994} provides quasi-isomorphic good-position subcomplexes, and the relation with motivic cohomology is described in \cite[Lectures~17--19]{MVW}.  Every graph intersection used below is checked directly to be admissible. For a finite defining system in good position, products are computed in the corresponding moving-compatible cycle model. Exterior products are first formed on separate factors and then mapped to the product variety; no motivic K\"unneth quasi-isomorphism is assumed. Proper pushforward, pullback between smooth varieties, and the projection formula are used in their standard higher Chow sense \cite[Section~1.3]{Bloch1986}\cite[Lectures~17--19]{MVW}.

\subsection{Derived coefficients and Lawson homotopy groups}\label{sec:coefficients}
For motivic or morphic cohomology, coefficients in a commutative ring $R$ are obtained by derived scalar extension at the level of the multiplicative enhancement.  If $A_X$ is a weight-graded integral cochain model, set
\[
 A_{X,R}=A_X\otimes^{\mathbf L}_{\mathbb Z}R,
 \qquad H^m_A(X;R(p)):=H^m(A_{X,R}(p)).
\]
Equivalently, choose a cofibrant associative representative whose underlying complex is $K$-flat before tensoring.  Here $K$-flat means that tensoring with an acyclic complex preserves acyclicity.  An arbitrary replacement of complexes without its multiplicative structure would not justify a Massey calculation; see \cite{Shipley2007,Stacks}.

For Lawson homology, restrict here to a complex projective variety $X$ and retain the original homotopy-theoretic definition.  Let $\mathcal Z_p(X)$ be Lawson's topological abelian group of algebraic $p$-cycles.  Then
\begin{equation}\label{eq:lawson-def}
 L_pH_k(X)=\pi_{k-2p}(\mathcal Z_p(X),0),\qquad k\ge2p\ge0.
\end{equation}
Let $\mathrm{Sing}_\bullet\mathcal Z_p(X)$ be its singular simplicial \emph{abelian} group; no free abelianization of singular simplices is being taken.  By the Moore/Dold--Kan theorem, the homology of its normalized complex is naturally the homotopy of $\mathcal Z_p(X)$; see May \cite{MaySimplicial}.  Let $\mathbf L_p(X)$ be the associated connective $\HZ$-module spectrum, where ``connective'' means that its negative homotopy groups vanish.  Thus
\[
 \pi_j\mathbf L_p(X)\cong\pi_j\mathcal Z_p(X),\qquad j\ge0.
\]
The passage between chain complexes and $\HZ$-module spectra is a chain of Quillen equivalences; see Schwede--Shipley \cite[Theorem~5.1.6]{SchwedeShipley}.  We define
\begin{equation}\label{eq:lawson-coeff-spectrum}
 L_pH_k(X;R):=\pi_{k-2p}\bigl(\mathbf L_p(X)\dsmash_{\HZ}\HR\bigr),
\end{equation}
where $\HR$ is the Eilenberg--Mac Lane ring spectrum of $R$ and $\dsmash_{\HZ}$ denotes the derived relative smash product.  Formula \eqref{eq:lawson-coeff-spectrum} is a coefficient convention on the homotopy theory of the cycle space, not a redefinition of Lawson homology as ordinary singular homology.

For higher Chow groups we similarly set
\[
 \CH^p(X,n;R):=H_n\bigl(z^p(X,\bullet)\otimes^{\mathbf L}_{\mathbb Z}R\bigr).
\]
Since $z^p(X,\bullet)$ is bounded below and degreewise free abelian, ordinary tensoring of that complex computes the derived tensor product; this does \emph{not} imply that taking higher Chow homology commutes with tensoring.

\begin{proposition}[Universal coefficient sequences]\label{prop:uct}
For every commutative ring $R$ there are natural short exact sequences in the ranges specified above
\begin{align}
0&\longrightarrow L_pH_k(X)\otimes R\longrightarrow L_pH_k(X;R)
 \longrightarrow \Tor^{\mathbb Z}_1(L_pH_{k-1}(X),R)\longrightarrow0,\label{eq:uct-lawson}\\
0&\longrightarrow L^pH^m(X)\otimes R\longrightarrow L^pH^m(X;R)
 \longrightarrow \Tor^{\mathbb Z}_1(L^pH^{m+1}(X),R)\longrightarrow0,\label{eq:uct-morphic}\\
0&\longrightarrow \CH^p(X,n)\otimes R\longrightarrow \CH^p(X,n;R)
 \longrightarrow \Tor^{\mathbb Z}_1(\CH^p(X,n-1),R)\longrightarrow0.\label{eq:uct-chow}
\end{align}
We put $L_pH_j(X)=0$ for $j<2p$ in \eqref{eq:uct-lawson}, and $\CH^p(X,-1)=0$ in \eqref{eq:uct-chow}.  All tensor products and $\Tor$ groups here are over $\Z$; $\Tor_1$ is the first derived functor of tensor product.
\end{proposition}
\begin{proof}
Apply the universal coefficient spectral sequence for the derived tensor product over $\mathbb Z$.  Since $\mathbb Z$ has global dimension one, only $\Tor_0$ and $\Tor_1$ occur. Equivalently, a length-one free resolution of the underlying abelian group of $R$ gives a two-column double complex; this also justifies the cohomological formula without an additional boundedness assumption.  In homological grading the adjacent term is one degree lower; in cohomological grading it is one degree higher.  See, for example, the Stacks Project \cite[Tag~061Y]{Stacks}.
\end{proof}

If $R$ is flat over $\mathbb Z$ (equivalently, its underlying abelian group is torsion-free), the Tor terms vanish.  Thus $\mathbb Z,\mathbb Q,\mathbb R,\mathbb C$, and $\mathbb Z_{(\ell)}$ give the expected tensor-product identifications.  For an integer $n\ge2$ and $R=\Z/n$, one obtains instead
\begin{align*}
0&\to L_pH_k(X)/n\to L_pH_k(X;\mathbb Z/n)
 \xrightarrow{\beta_L}L_pH_{k-1}(X)[n]\to0,\\
0&\to L^pH^m(X)/n\to L^pH^m(X;\mathbb Z/n)
 \xrightarrow{\beta_{\mor}}L^pH^{m+1}(X)[n]\to0.
\end{align*}
The connecting maps $\beta_L$ and $\beta_{\mor}$ are the corresponding Bockstein homomorphisms.  Hence quotienting an integral group by $n$ must not be confused with the full finite-coefficient theory.

\begin{remark}[A concrete Lawson example]
For a complex Enriques surface $S$, one has $H^1(S;\Z)=0$ and $\NS(S)[2]\cong\Z/2$; see \cite[Chapter~VIII]{BPV}.  Here $\NS(S)=\Pic(S)/\Pic^0(S)$ is the N\'eron--Severi group and $\Pic^0(S)$ is the identity component of the Picard scheme.  Friedlander--Lawson's weight-one calculation gives $L^1H^1(S)=0$ and $L^1H^2(S)=\NS(S)$; see  \cite{FLCocycles}.  Duality then gives $L_1H_3(S)=0$ and $L_1H_2(S)=\NS(S)$.  Consequently \eqref{eq:uct-lawson} yields
\[
 L_1H_3(S;\mathbb Z/2)\cong\mathbb Z/2,
 \qquad L_1H_3(S)\otimes\mathbb Z/2=0.
\]
This is exactly the kind of Tor contribution that is lost if finite coefficients are replaced by tensoring the already computed Lawson group.
\end{remark}

At higher Chow degree zero, the last term of \eqref{eq:uct-chow} vanishes for every $R$, so
\begin{equation}\label{eq:ordinary-chow-coeff}
 \CH^p(X,0;R)\cong\CH^p(X)\otimes R.
\end{equation}
Likewise $L_pH_{2p}(X;R)\cong L_pH_{2p}(X)\otimes R$. These lowest-degree facts do not extend to arbitrary higher degrees.

For an integral triple of ordinary Chow classes, put $s=p+q+r$, let $I_M\subset\CH^s(X,1)$ be \eqref{eq:intro-indet}, and let $I_{M,R}$ be the full indeterminacy in $M_X\otimes^{\mathbf L}_{\Z}R$. If $R$ is $\Z$-flat, the multiplication images defining these subgroups give canonical identifications
\begin{equation}\label{eq:flat-quotient}
 I_{M,R}=\im\bigl(I_M\otimes R\bigr),\qquad
 \bigl(\CH^s(X,1)/I_M\bigr)\otimes R
 \cong\CH^s(X,1;R)/I_{M,R}.
\end{equation}
Indeed, flatness commutes with cohomology, images, and quotients, and each coefficient class is a finite sum of tensors of integral classes. The transported defining system gives the corresponding quotient class. For nonflat $R$, the full coefficient groups, including their $\Tor$ contributions, must be used in both indeterminacy summands, and \eqref{eq:flat-quotient} is not asserted. No canonical splitting of the universal coefficient sequences is needed. All nonvanishing assertions below concern integral brackets or the specified flat localizations.

\section{The Chow--Massey construction and its indeterminacy}
The construction and its invariance are classical \cite{May1969}. We record the calculation proving Theorem~\ref{thm:intro-construction} in the present grading. Choose closed representatives
\[
 A\in M_X^{2p}(p),\qquad B\in M_X^{2q}(q),\qquad C\in M_X^{2r}(r)
\]
for $a,b,c$.  Since $ab=bc=0$, choose cochains
\[
 U\in M_X^{2(p+q)-1}(p+q),\qquad V\in M_X^{2(q+r)-1}(q+r)
\]
with
\begin{equation}\label{eq:nulls}
 dU=AB,\qquad dV=BC.
\end{equation}
Because $A,B,C$ have even cohomological degree,
\begin{equation}\label{eq:triple-cocycle}
 W=AV-UC,
 \qquad dW=A(BC)-(AB)C=0.
\end{equation}
Its weight is $s=p+q+r$ and its degree is $2s-1$, so
\[
 [W]\in H_M^{2s-1}(X,\mathbb Z(s))=\CH^s(X,1).
\]

\begin{definition}
The \emph{Chow--Massey product} $\langle a,b,c\rangle_{\CH}$ is the set of all cohomology classes represented by \eqref{eq:triple-cocycle} as the representatives and null-homotopies vary in the fixed multiplicative motivic enhancement.  Equivalently, it is the corresponding motivic Toda bracket expressed through \eqref{eq:motivic-higher Chow}.
\end{definition}

For fixed $A,B,C$, replacing $U,V$ by $U+u,V+v$ with $du=dv=0$ changes the value by $a[v]-[u]c$, giving precisely the subgroup \eqref{eq:intro-indet}.  Conversely every element of that subgroup is obtained this way.  Independence from cocycle representatives follows directly: if $A$ is changed to $A+dS$, replace $U$ by $U+SB$ and observe
\[
 dS\,V-SBC=d(SV),
\]
while if $C$ is changed to $C+dT$, replace $V$ by $V+BT$ and obtain
\[
 ABT-UdT=d(UT).
\]
If $B$ is changed to $B+dK$, replace $U,V$ by $U+AK,V+KC$; the two changes in $W$ cancel.  This proves the triple assertion of Theorem~\ref{thm:intro-construction}.

\begin{proposition}[Naturality and multiplicative invariance]\label{prop:naturality}
A weight-preserving multiplicative map of enhancements sends a defined triple bracket into the bracket of the image classes and sends the full source indeterminacy into the full target indeterminacy.  A multiplicative quasi-isomorphism identifies the triple brackets under the induced cohomology isomorphism.
\end{proposition}
\begin{proof}
Apply the map to a defining system; see \cite{May1969}.  For a quasi-isomorphism, the two full indeterminacies are identified.  The image of the source bracket and the target bracket are then nonempty cosets of the same subgroup.  Naturality gives an inclusion, hence equality.  This proves invariance along zigzags as well.  The spectrum-level formulation is the naturality of Toda brackets.
\end{proof}

Here and below the degree of a cochain is its cohomological degree, not its weight. For homogeneous motivic classes of degrees $d_1,d_2,d_3$, the sign convention is
\[
 W=AV+(-1)^{d_1+1}UC,\qquad dU=AB,\quad dV=BC.
\]
More generally, for $a_i\in\CH^{p_i}(X,j_i)$ put $d_i=2p_i-j_i$.  An entry spanning the classes indexed by $i,\ldots,k$ in May's defining system has weight $\sum_{t=i}^k p_t$ and cohomological degree $\sum_{t=i}^k d_t-(k-i)$.  Consequently
\[
 \langle a_1,\ldots,a_N\rangle_{\CH}
 \subset\CH^{\sum p_i}\left(X,\sum j_i+N-2\right)
\]
whenever a full defining system exists \cite{May1969}.  For $N\ge4$, vanishing adjacent products alone is insufficient.  We assert no analogue of the simple affine triple indeterminacy for all higher brackets.  All explicit examples below are triple products.

\section{Torsion divisors on a curve and the Weil pairing}\label{sec:curve}
Let $C$ be a smooth connected complex projective curve and $n\ge2$.  Choose degree-zero divisors
\[
 D=\sum_Pm_PP,\qquad D'=\sum_Qm'_QQ
\]
with disjoint supports and rational functions $f,g\in\mathbb C(C)^\times$ satisfying
\begin{equation}\label{eq:divisors}
 \operatorname{div}(f)=nD,\qquad \operatorname{div}(g)=nD'.
\end{equation}
Write $\alpha=[D]$ and $\gamma=[D']$ in $\CH^1(C)=\Pic(C)$.  Any pair of $n$-torsion classes in $\Pic^0(C)$ can be represented this way after moving the divisors to have disjoint support.

For a divisor $A=\sum a_PP$ whose support avoids the zeros and poles of a rational function $h$, write
\[
 h(A):=\prod_Ph(P)^{a_P}.
\]
Let $J(C)=\Pic^0(C)$ be the Jacobian.  Its canonical principal polarization gives the perfect Weil pairing
\[
 e_n:J(C)[n]\times J(C)[n]\longrightarrow\mu_n,
\]
where $\mu_n\subset\mathbb C^\times$ is the group of $n$th roots of unity.
With the convention \eqref{eq:divisors} and disjoint supports, we use
\begin{equation}\label{eq:weil-convention}
 e_n(\alpha,\gamma)=\frac{f(D')}{g(D)}.
\end{equation}
The equivalence between the polarization-theoretic Weil pairing and this divisor formula is classical; a convenient reference is Howe \cite{Howe1996}.

\begin{theorem}[Curve linking bracket]\label{thm:curve-linking}
The bracket
\[
 \langle\alpha,n,\gamma\rangle_{\CH}\subset\CH^2(C,1),
\]
where the middle class is $n\cdot1_C\in\CH^0(C)$, is a singleton $\{\theta\}$.  Write $\Gamma_f$ and $\Gamma_g$ for the graphs in $C\times\mathbb P^1$ restricted to $C\times\square$, viewed as elements of $z^1(C,1)$. We write $\theta=[\Theta]$ for a homology class of a cycle $\Theta\in z^2(C,1)$. A term $[P,u]$ means the closed point $(P,u)$ with $u\in\C^\times\setminus\{1\}$.  Then $\theta$ has the representative
\begin{equation}\label{eq:theta}
 \Theta=D\cdot\Gamma_g-\Gamma_f\cdot D'
 =\sum_Pm_P[P,g(P)]-\sum_Qm'_Q[Q,f(Q)].
\end{equation}
Terms with coordinate $1$ are omitted.  Moreover,
\begin{equation}\label{eq:push-weil}
 n\theta=0,
 \qquad (\pi_C)_*\theta=\frac{g(D)}{f(D')}=e_n(\alpha,\gamma)^{-1}
 \in\mathbb C^\times=\CH^1(\Spec\mathbb C,1),
\end{equation}
where $\pi_C:C\to\Spec\mathbb C$ is the structure morphism.  If $e_n(\alpha,\gamma)$ has exact order $n$, then $\theta$ has exact order $n$.
\end{theorem}
\begin{proof}
The graph convention gives $\partial\Gamma_f=nD$ and $\partial\Gamma_g=nD'$. If $D$ is the zero divisor, one may take $f=1$, and similarly for $D'$; the graph of the constant $1$ is the zero chain in the open cube.  In the moving-compatible motivic enhancement, these are the two rational-equivalence null-homotopies.  The middle class $n\cdot1$ acts by scalar multiplication in the unital motivic model.  The only nonscalar products required are $D\cdot\Gamma_g$ and $\Gamma_f\cdot D'$, and they are in good position by the disjoint-support hypothesis.  Thus the standard cycle product computes this defining system \cite{Bloch1994,MVW}, and \eqref{eq:triple-cocycle} gives \eqref{eq:theta}.  No multiplication of arbitrary unmoved cycles, or rational alternating complex, is being used.  The disjoint-support condition makes all displayed intersections proper and ensures that the point coordinates avoid $0$ and $\infty$; a term with coordinate $1$ is absent from the open cube.

Let $\Gamma_{(f,g)}\in z^2(C,2)$ be the graph of $(f,g)$ in $C\times\square^2$.  Because the zero and pole divisors of $f$ and $g$ are disjoint, it meets all faces properly and no codimension-two face.  Its cubical boundary is
\begin{equation}\label{eq:twotheta}
 \partial\Gamma_{(f,g)}=nD\cdot\Gamma_g-\Gamma_f\cdot nD'=n\Theta.
\end{equation}
Thus $n\theta=0$ in the actual higher Chow group.

For a smooth projective connected curve, $\CH^1(C,1)=\Gamma(C,\mathcal O_C^\times)=\mathbb C^\times$; see \cite[Corollary~4.2]{MVW}.  This group is divisible. We write $[u]$ for the corresponding weight-one motivic class of a constant unit $u$, so $[uv]=[u]+[v]$. Hence, for $u\in\mathbb C^\times$ and $v^n=u$,
\[
 \alpha[u]=\alpha\,n[v]=(n\alpha)[v]=0,
 \qquad [u]\gamma=[v](n\gamma)=0.
\]
Both summands of the full indeterminacy therefore vanish, so the bracket is a singleton.

Proper pushforward sends $[P,u]$ to $[u]$, and the additive law in $\CH^1(\Spec\mathbb C,1)$ corresponds to multiplication in $\mathbb C^\times$.  Thus \eqref{eq:theta} gives the ratio in \eqref{eq:push-weil}.  Formula \eqref{eq:weil-convention} identifies it with the inverse Weil pairing.  If this root of unity has exact order $n$, then \eqref{eq:twotheta} and the pushforward homomorphism together imply that $\theta$ has exact order $n$.
\end{proof}

\begin{remark}[Relation with earlier work]
This calculation should be read together with the biextension literature.  Bloch's construction of Chow biextensions \cite{Bloch1989} and Gorchinskiy's comparison theorem already connect the relevant pairings to higher Chow complexes; Gorchinskiy further explains the associated Weil pairing through a Massey triple product and proves the corresponding $K$-cohomological formula \cite[Example~3.10, Section~4.1, and Corollary~4.36]{GorchinskiyBiext}\cite[Propositions~4.25 and~4.27]{Gorchinskiy}.  Thus Theorem~\ref{thm:curve-linking} is a concrete cubical specialization and full indeterminacy calculation.
\end{remark}

\begin{corollary}[All torsion orders on every positive genus curve]\label{cor:all-curves}
Let $C$ be a smooth connected complex projective curve of genus $g\ge1$.  For every $n\ge2$ there are classes $\alpha,\gamma\in\Pic^0(C)[n]$ such that
\[
 \langle\alpha,n,\gamma\rangle_{\CH}=\{\theta\}\subset\CH^2(C,1)
\]
is a singleton whose unique value has exact order $n$.
\end{corollary}
\begin{proof}
The Weil pairing on $J(C)[n]\cong(\Z/n)^{2g}$ is perfect \cite{Howe1996}. Since $g\ge1$, choose an element $\alpha$ of exact order $n$; the associated character $e_n(\alpha,-)$ has exact order $n$, so some $\gamma$ satisfies that $e_n(\alpha,\gamma)$ is a primitive $n$th root of unity.  Move divisor representatives to disjoint support and apply Theorem~\ref{thm:curve-linking}.
\end{proof}

\section{An explicit elliptic calculation}
Let $E$ be the smooth projective model of
\[
 E:y^2=x^3-x
\]
with point at infinity $O$, and write $\pi_E:E\to\Spec\C$ for its structure morphism.  Put
\[
 P=(0,0),\qquad Q=(1,0),\qquad T=(2,\sqrt6).
\]
The points $P,Q$ are nonzero $2$-torsion points.  The chord through $T$ and $Q$ has slope $\sqrt6$, so the group law gives
\[
 T'=T+Q=(3,-2\sqrt6).
\]
Define
\begin{equation}\label{eq:explicit-divisors}
 D=P-O,\qquad D'=T'-T,\qquad f=x,\qquad g(S)=x(S-T)-1.
\end{equation}
Then
\[
 \operatorname{div}(f)=2P-2O,
 \qquad \operatorname{div}(g)=2T'-2T.
\]
The four support points are distinct.  A direct addition calculation gives
\[
 x(P-T)=-\frac12,
 \qquad g(P)=-\frac32,
 \qquad g(O)=1,
 \qquad f(T')=3,
 \qquad f(T)=2.
\]

\begin{proposition}[A three-point higher cycle]\label{prop:explicit}
For $\alpha=[P-O]$ and $\gamma=[T'-T]$, put
\[
 \Theta_E=[P,-3/2]-[T',3]+[T,2]\in z^2(E,1),
 \qquad \theta_E=[\Theta_E]\in\CH^2(E,1).
\]
Then $\langle\alpha,2\cdot1_E,\gamma\rangle_{\CH}=\{\theta_E\}$. Its full indeterminacy is zero and $\theta_E$ has exact order two.
\end{proposition}
\begin{proof}
Substitution in \eqref{eq:theta} gives $\Theta_E$; the term $[O,1]$ is absent from $\square$. Proper pushforward gives
\[
 (\pi_E)_*\theta_E=(-3/2)\,3^{-1}\,2=-1\ne1
 \quad\text{in }\C^\times.
\]
Thus $\theta_E\ne0$. Equation~\eqref{eq:twotheta} gives $2\theta_E=0$, and Theorem~\ref{thm:curve-linking} gives zero full indeterminacy.
\end{proof}

\section{Geometric applications}\label{sec:applications}
\subsection{Projective products}
Let $C,D,D',n,\theta$ be as in Theorem~\ref{thm:curve-linking}, with $\theta$ of exact order $n$.  For $r\ge1$ set
\[
 X_r=C\times\mathbb P^r,
 \qquad p:X_r\to C,
 \qquad h=c_1(\mathcal O_{\mathbb P^r}(1))\in\CH^1(X_r),
\]
where pullbacks of $h$ are implicit.

\begin{theorem}[Three divisor classes]\label{thm:projective-transfer}
Set
\[
 a=p^*\alpha,\qquad b=nh,\qquad c=p^*\gamma\in\CH^1(X_r).
\]
Then $ab=bc=0$ and
\[
 \langle a,b,c\rangle_{\CH}=\{\eta_r\},
 \qquad \eta_r=h\,p^*\theta\in\CH^3(X_r,1).
\]
The full indeterminacy is zero and $\eta_r$ has exact order $n$.
\end{theorem}
\begin{proof}
The adjacent products vanish because $n\alpha=n\gamma=0$.  If $H\subset\mathbb P^r$ is a hyperplane, the chains $\Gamma_f\boxt H$ and $\Gamma_g\boxt H$ have the required boundaries, and their defining-system value is $\Theta\boxt H$.  All intersections on the curve are those already checked in Theorem~\ref{thm:curve-linking}; the projective-space factor is external, so no motivic K\"unneth isomorphism is being assumed.

The projective bundle formula gives
\begin{equation}\label{eq:pb-formula}
 \CH^2(X_r,1)
 \cong p^*\CH^2(C,1)\oplus h\,p^*\CH^1(C,1),
\end{equation}
with further possible summands zero because $\CH^0(C,1)=0$ and negative codimension groups vanish; see \cite[Theorem~15.12]{MVW}.  Multiplication by $a$ or $c$ kills the first summand because $\CH^3(C,1)=0$ (there are no codimension-three cycles in $C\times\square$), and it kills the second by the torsion/divisibility calculation in Theorem~\ref{thm:curve-linking}.  Thus the full indeterminacy
\[
 a\,\CH^2(X_r,1)+\CH^2(X_r,1)c
\]
is zero.

Finally the projective bundle pushforward and projection formula yield
\[
 p_*(h^{r-1}\eta_r)=p_*(h^rp^*\theta)=\theta.
\]
Hence $\eta_r$ has order at least $n$, while $n\eta_r=0$ follows from $n\theta=0$.
\end{proof}

The surface $C\times\mathbb P^1$ already has a nonzero bracket in $\CH^3(-,1)$ with three ordinary divisor classes.  For comparison with ordinary nonnegative-index Lawson homology, the threefold $C\times\mathbb P^2$ is more convenient.

\begin{corollary}[A nonzero realization-kernel class in every torsion order]\label{cor:kernel-family}
Let $C$ have genus $g\ge1$ and let $n\ge2$.  On $Y=C\times\mathbb P^2$ there is a triple of divisor classes whose Chow--Massey quotient is a nonzero element of exact order $n$, while its image in the morphic Massey quotient is zero.  The relevant morphic target is $L^3H^5(Y)\cong\mathbb Z^{2g}$.
\end{corollary}
\begin{proof}
Choose $\alpha,\gamma$ by Corollary~\ref{cor:all-curves}, apply Theorem~\ref{thm:projective-transfer} with $r=2$, and then apply Theorem~\ref{thm:loss} below.  The target computation is given after that theorem.
\end{proof}

\begin{corollary}[Every dimension in the Lawson range]\label{cor:all-dimensions}
For every $d\ge3$ and $n\ge2$, the $d$-fold $Y_d=C\times\mathbb P^{d-1}$, with $g(C)\ge1$, has a bracket of three divisor classes whose motivic indeterminacy is zero and whose value has exact order $n$.  Its morphic and singular brackets are $\{0\}$.  The morphic target $L^3H^5(Y_d)$ is nonzero and is identified by duality with $L_{d-3}H_{2d-5}(Y_d)$.
\end{corollary}
\begin{proof}
Take $r=d-1$ in Theorem~\ref{thm:projective-transfer}, whose projective-bundle calculation already proves vanishing of the \emph{entire} motivic indeterminacy.  The outer divisor classes lie in $\Pic^0(Y_d)$, so their morphic and singular images vanish.  The realized bracket is therefore $\langle0,n\rho h,0\rangle=\{0\}$, as proved in Theorem~\ref{thm:loss}.  To see that its target is nonzero, choose $0\ne u\in L^1H^1(C)=H^1(C^{\mathrm{an}};\Z)$.  The class $u\boxtimes\rho(h)^2\in L^3H^5(Y_d)$ has nonzero singular realization by the integral K\"unneth theorem.  The degree conversion is $(3,5)\mapsto(d-3,2d-5)$.
\end{proof}

\subsection{Products of curves and ample canonical bundles}
The projective factor can be replaced by a positive genus curve.  The full indeterminacy may then be larger; the following proof uses an additive detector instead of asserting that it vanishes.
\begin{theorem}[A Projection formula homomorphism on a product of curves]\label{thm:ample-products}
Let $d\ge2$ and let $C_1,\ldots,C_d$ be smooth connected complex projective curves, with $g(C_1)\ge1$.  Put $Z=\prod_{i=1}^d C_i$ and let $p_i:Z\to C_i$ be the projections.  For every $n\ge2$ there are divisor classes $a,b,c\in\CH^1(Z)$ with a specified value
\[
 \eta_Z\in\langle a,b,c\rangle_{\CH}\subset\CH^3(Z,1)
\]
such that both $\eta_Z$ and its class modulo the full indeterminacy have exact order $n$.  The realized morphic and singular brackets are $\{0\}$.  If all $g(C_i)\ge2$, then $K_Z$ is ample.  For $d\ge3$, $L^3H^5(Z)$ is nonzero and has an ordinary Lawson interpretation.
\end{theorem}
\begin{proof}
Choose $\alpha,\gamma,\theta$ on $C_1$ as in Corollary~\ref{cor:all-curves}.  For $i\ge2$ choose $P_i\in C_i(\C)$ and put $h_i=p_i^*[P_i]\in\CH^1(Z)$.  Set
\[
 a=p_1^*\alpha,\qquad b=nh_2,\qquad c=p_1^*\gamma,
 \qquad \eta_Z=p_1^*\theta\,h_2.
\]
The factorwise defining system of Theorem~\ref{thm:projective-transfer}, now with the point divisor $P_2$ in place of a hyperplane, gives this value.  The remaining factors carry their fundamental classes.  It follows in particular that $n\eta_Z=0$.

Let $t=h_3\cdots h_d$, with $t=1$ if $d=2$, and define the additive maps
\[
 T_s:\CH^s(Z,1)\longrightarrow\CH^{s-1}(C_1,1),
 \qquad T_s(v)=(p_1)_*(t\,v).
\]
These maps lower codimension by one: multiplication by $t$ raises it by $d-2$ and pushforward lowers it by $d-1$.  Proper pushforward and the projection formula in higher Chow theory \cite{Bloch1986,MVW} give
\[
 T_3(\eta_Z)=\theta,\qquad
 T_3(a v)=\alpha\,T_2(v),\qquad T_3(v c)=T_2(v)\,\gamma.
\]
Since $\CH^1(C_1,1)=\C^\times$, the torsion--divisibility calculation in Theorem~\ref{thm:curve-linking} shows that $T_3$ kills
\[
 I_Z=a\,\CH^2(Z,1)+\CH^2(Z,1)\,c,
\]
the complete indeterminacy.  Thus $k\eta_Z\in I_Z$ implies $k\theta=0$, which forces $n\mid k$.  Together with $n\eta_Z=0$, this proves both exact-order assertions.  We do not assert $I_Z=0$.

The two outer divisor classes belong to $\Pic^0(Z)$, so their morphic images vanish and the entire realized bracket is $\{0\}$.  The canonical bundle is $K_Z=\bigotimes_i p_i^*K_{C_i}$.  If every genus is at least two, all $K_{C_i}$ are ample; a common tensor power and the Segre embedding show that their exterior tensor product is ample.

For $d\ge3$, take $0\ne u\in L^1H^1(C_1)=H^1(C_1;\Z)$.  The morphic class $p_1^*u\,\rho(h_2)\rho(h_3)$ has nonzero singular realization by K\"unneth, so $L^3H^5(Z)\ne0$.  Duality gives $L^3H^5(Z)\cong L_{d-3}H_{2d-5}(Z)$.  This proves Theorem~\ref{thm:intro-ample}.
\end{proof}

This homomorphism is not multiplicative and is not being used to transport a Massey bracket by naturality.  It is used only with the projection formula on the two complete indeterminacy summands.  The application concerns disappearance at the motivic-to-morphic step; it is different from the ample canonical examples of \cite[Theorem~1.6]{HuLawson}, where the secondary values survive at the Lawson stage.

\subsection{Coprime degree pullback and cyclic covers}
The following is the higher Chow form of the usual transfer argument; compare the morphic version in \cite[Proposition~12.1]{HuLawson}. Only the projection formula is required, not multiplicativity of pushforward.
\begin{proposition}[Coprime degree transfer]\label{prop:coprime-transfer}
Let $f:W\to V$ be a surjective generically finite morphism of degree $e$ between smooth connected complex projective varieties of the same dimension. Suppose that
\[
 w\in\langle a,b,c\rangle_{\CH}\subset\CH^s(V,1),
 \qquad s=p+q+r,
\]
where $a,b,c$ have codimensions $p,q,r$, respectively. Assume that $nw=0$ and that the class of $w$ modulo the full indeterminacy $I$ has exact order $n\ge2$. If $\gcd(e,n)=1$, then $f^*w$ and its class modulo the full indeterminacy $I'$ of the pulled-back triple both have exact order $n$.

If $\rho(a)=\rho(c)=0$, the entire morphic and singular brackets of the pulled-back triple are $\{0\}$. Whenever the ordinary Lawson range holds, its Lawson bracket is $\{0\}$ as well.
\end{proposition}
\begin{proof}
Pullback is multiplicative, so $f^*w$ is a value of the pulled-back bracket. Proper pushforward and the projection formula in higher Chow groups give
\begin{equation}\label{eq:transfer-full}
 f_*I'\subseteq I,\qquad f_*f^*w=e\,w.
\end{equation}
Indeed, for $v\in\CH^{q+r}(W,1)$ one has
$f_*((f^*a)v)=a f_*v$, and the other indeterminacy summand is treated in the same way. Since $f_*1_W=e1_V$, the projection formula also gives the second equality. These are standard functorialities of higher Chow theory \cite[Section~1.3]{Bloch1986}\cite[Lectures~17--19]{MVW}.

If $k f^*w\in I'$ for an integer $k$, then \eqref{eq:transfer-full} implies $ke\,w\in I$. Thus $n\mid ke$, and coprimality gives $n\mid k$. Since $n f^*w=0$, the quotient class has exact order $n$, and the same is true of the value itself. If the two outer morphic classes vanish on $V$, they still vanish after pullback. A bracket with both outer classes zero has zero full indeterminacy and contains zero, so it is the singleton $\{0\}$ in each realized theory.
\end{proof}

\begin{corollary}[Birational morphisms and projective bundles]\label{cor:bundle-birational}
Each specified nonzero value in Corollary~\ref{cor:all-curves} and Theorems~\ref{thm:projective-transfer} and~\ref{thm:ample-products} retains its exact order, both before and after quotienting by the full indeterminacy, under a projective birational morphism from a smooth projective variety. This applies in particular to a blowup along a smooth center. The same assertions hold after pullback along a projective bundle $q:\mathbb P(\mathcal E)\to V$, even without a section.
\end{corollary}
\begin{proof}
A birational morphism has degree one, so Proposition~\ref{prop:coprime-transfer} applies. For a vector bundle $\mathcal E$ of rank $r+1$, let $\zeta=c_1(\mathcal O_{\mathbb P(\mathcal E)}(1))$. The projective bundle formula \cite[Theorem~15.12]{MVW} gives $q_*(\zeta^r)=1_V$. The homomorphism
\[
 T(v):=q_*(\zeta^r v)
\]
preserves codimension and higher Chow degree, satisfies the projection formula, and obeys $Tq^*=\mathrm{id}$. Replacing $f_*$ by $T$ in \eqref{eq:transfer-full} proves the assertion. The outer realized classes remain zero. No vanishing of the new motivic indeterminacy is asserted, and no statement about arbitrary birational rational maps is implied.
\end{proof}

\begin{corollary}[Smooth cyclic covers with ample canonical bundle]\label{cor:cyclic-cover}
Fix $d\ge2$, $n\ge2$, and $e\ge2$ with $\gcd(e,n)=1$. There is a smooth connected projective $d$-fold $W$ with ample canonical bundle and a finite flat cyclic morphism
\[
 f:W\longrightarrow C\times\mathbb P^{d-1}
\]
of degree $e$, where $g(C)\ge1$, such that three divisor classes on $W$ define a motivic bracket with a specified value and a full indeterminacy quotient class of exact order $n$. Its entire morphic and singular brackets are $\{0\}$. For $d\ge3$, the morphic target $L^3H^5(W)$ is nonzero and is dual to $L_{d-3}H_{2d-5}(W)$.
\end{corollary}
\begin{proof}
Put $V=C\times\mathbb P^{d-1}$. Choose a line bundle $\mathcal L$ sufficiently positive that $\mathcal L^{\otimes e}$ is very ample and $K_V\otimes\mathcal L^{\otimes(e-1)}$ is ample. Let $s$ be a general section of $\mathcal L^{\otimes e}$ with nonempty smooth zero divisor $D$. The cyclic cover is
\[
 W=\Spec_V\!\left(\bigoplus_{i=0}^{e-1}\mathcal L^{-i}\right),
\]
where multiplication is determined by $s$. The standard cyclic-cover construction and ramification formula give smoothness and
\begin{equation}\label{eq:canonical-cyclic}
 K_W\cong f^*\!\left(K_V\otimes\mathcal L^{\otimes(e-1)}\right);
\end{equation}
see \cite[Section~3.5 and Lemmas~3.15--3.16]{EsnaultViehweg}. For clarity, connectedness follows at the generic point of any component of $D$: after trivializing $\mathcal L$ there, the polynomial $t^e-s$ is Eisenstein for that discrete valuation, because $D$ is reduced. Thus the cover is integral. Its finiteness gives projectivity and implies that the line bundle in \eqref{eq:canonical-cyclic} is ample.

Theorem~\ref{thm:projective-transfer}, with $r=d-1$, supplies the triple on $V$. Proposition~\ref{prop:coprime-transfer} preserves its exact-order quotient and gives the realized singleton brackets. When $d\ge3$, choose $0\ne u\in L^1H^1(C)$ and write $h_{\sing}=\mathrm{cl}\rho(h)$ for the singular hyperplane class on $\mathbb P^{d-1}$. The infinite-order class $\mathrm{cl}(u)\boxt h_{\sing}^2$ remains nonzero under $f^*$: singular-cohomology pushforward satisfies $f_*f^*=e\,\mathrm{id}$. Its morphic lift therefore remains nonzero. The stated Lawson degree follows from duality.
\end{proof}

Unlike Theorem~\ref{thm:ample-products}, this corollary may choose the cover after $n$ is fixed. For one cover of degree $e$, the transfer argument applies to orders prime to $e$; it makes no assertion for the remaining orders. The conclusion supplies smooth branched covers, not an assertion that every such cover is nonisomorphic to a product.

\section{Comparison and invariants beyond Lawson--Massey}\label{sec:realization}
\subsection{The comparison maps and their full indeterminacies}
Heller's monoidal comparison \cite[Theorem~4.11(2)]{Heller}, together with Proposition~\ref{prop:naturality}, gives
\begin{equation}\label{eq:bracket-realization}
 \begin{aligned}
 \rho\bigl(\langle a,b,c\rangle_{\CH}\bigr)&\subseteq
 \langle\rho a,\rho b,\rho c\rangle_{\mor},\\
 \mathrm{cl}\bigl(\langle\rho a,\rho b,\rho c\rangle_{\mor}\bigr)&\subseteq
 \langle\mathrm{cl}\rho a,\mathrm{cl}\rho b,\mathrm{cl}\rho c\rangle_{\sing}.
 \end{aligned}
\end{equation}
The last bracket is the classical singular-cohomology Massey bracket.  For $a,b,c$ of Chow codimensions $p,q,r$, respectively, its indeterminacy and the morphic indeterminacy are
\begin{align}
 I_{\mor}&=(\rho a)L^{q+r}H^{2(q+r)-1}(X)
       +L^{p+q}H^{2(p+q)-1}(X)(\rho c),\label{eq:morphic-indet}\\
 I_{\sing}&=(\mathrm{cl}\rho a)H^{2(q+r)-1}(X^{\mathrm{an}};\Z)
       +H^{2(p+q)-1}(X^{\mathrm{an}};\Z)(\mathrm{cl}\rho c).
       \label{eq:singular-indet}
\end{align}
Thus $\rho(I_M)\subseteq I_{\mor}$ and $\mathrm{cl}(I_{\mor})\subseteq I_{\sing}$.  Naturality alone gives inclusions of brackets, not equality: the target may have additional null-homotopies.

For a smooth projective $d$-fold, write
\[
 D_X:L^uH^v(X)\xrightarrow{\sim}L_{d-u}H_{2d-v}(X)
\]
for Friedlander--Lawson duality and $\operatorname{PD}_X$ for singular Poincar\'e duality \cite{FLDuality}.  They satisfy
\[
 \Phi\circ D_X=\operatorname{PD}_X\circ\mathrm{cl}.
\]
The Lawson--Massey bracket in \cite{HuLawson} is
\[
 \langle x_1,x_2,x_3\rangle_L
 :=D_X\langle D_X^{-1}x_1,D_X^{-1}x_2,D_X^{-1}x_3\rangle_{\mor}.
\]
For $x_i\in L_{p_i}H_{k_i}(X)$, its target and full indeterminacy are
\begin{align}
 &L_{p_1+p_2+p_3-2d}H_{k_1+k_2+k_3-4d+1}(X),\label{eq:lawson-target}\\
 I_L={}&x_1\bullet L_{p_2+p_3-d}H_{k_2+k_3-2d+1}(X)
       +L_{p_1+p_2-d}H_{k_1+k_2-2d+1}(X)\bullet x_3.
       \label{eq:lawson-indet}
\end{align}
Here $\bullet$ is the intersection product transferred from morphic cohomology.  These formulas are used only when every displayed group is in the ordinary Lawson range: its cycle index is nonnegative and its homological degree is at least twice that index. Equivalently, a morphic group $L^uH^v$ is transferred here only when $0\le u\le d$ and $v\le2u$.  The singular-homology bracket is likewise the singular-cohomology bracket transported by $\operatorname{PD}_X$.  Its indeterminacy is $\operatorname{PD}_X(I_{\sing})$, in the corresponding degrees.

The ordinary Chow-to-morphic comparison on diagonal groups is
\begin{equation}\label{eq:diag-morphic}
 L^pH^{2p}(X)\cong\CH^p(X)/\CH^p(X)_{\mathrm{alg}},
 \qquad
 \Pic(X)\longrightarrow L^1H^2(X)=\NS(X),
\end{equation}
where $\CH^p(X)_{\mathrm{alg}}$ is the subgroup algebraically equivalent to zero; see \cite{FLCocycles,FLDuality}.  In particular, the second map is the quotient by $\Pic^0(X)$.  This differs from the rational-equivalence relation defining $\CH^p(X)$.

\begin{theorem}[Complete loss of the curve bracket]\label{thm:loss}
For every bracket of Theorem~\ref{thm:projective-transfer}, and also for the brackets of Theorem~\ref{thm:ample-products}, the outer morphic classes are zero.  The entire morphic and singular-cohomology brackets are $\{0\}$.  Whenever the target has a nonnegative Lawson index, the associated Lawson bracket is also $\{0\}$.
\end{theorem}
\begin{proof}
Both outer divisor classes are pulled back from $\Pic^0(C)$, so \eqref{eq:diag-morphic} kills them.  A triple $\langle0,b',0\rangle$ contains zero by taking the two outer representatives and both null-homotopies to be zero.  Its two indeterminacy summands are also zero, so it is exactly $\{0\}$.  The argument applies to morphic and singular cohomology; duality gives the Lawson statement.  Notice that this proves a singleton assertion from the vanishing of the outer classes, not from naturality alone.
\end{proof}

For $Y=C\times\mathbb P^2$, duality and Dold--Thom \cite{Lawson1989,FLDuality} give
\[
 L^3H^5(Y)\cong L_0H_1(Y)\cong H_1(Y^{\mathrm{an}};\Z)\cong\Z^{2g(C)}.
\]
Also $H^5(Y^{\mathrm{an}};\Z)=H^1(C;\Z)\boxtimes h_{\sing}^2\cong\Z^{2g(C)}$, where $h_{\sing}=\mathrm{cl}\rho(h)$.  Thus the vanishing in Theorem~\ref{thm:loss} is not explained by zero target groups.

\subsection{Failure of recovery from the realized triple}
\begin{theorem}[A cyclic family with identical Lawson--Massey realization]\label{thm:fixed-realized-triple}
For fixed $C$, $n$, and $Y=C\times\mathbb P^2$ as in Theorem~\ref{thm:intro-loss-family}, one can choose $a,b,c$ and $\eta$ such that, for $0\le k<n$,
\begin{equation}\label{eq:fixed-family}
 \langle a,b,kc\rangle_{\CH}=\{k\eta\},\qquad
 I_M(a,b,kc)=0,\qquad
 (\rho a,\rho b,\rho(kc))=(0,n\rho h,0).
\end{equation}
The $n$ values are distinct; precisely $n-1$ are nonzero.  The subgroup they form is a copy of $\Z/n$ in
$\ker(\rho:\CH^3(Y,1)\to L^3H^5(Y))$.  The corresponding morphic and Lawson brackets are all $\{0\}$.
\end{theorem}
\begin{proof}
Choose $\alpha,\gamma,\theta$ of exact order $n$ as in Corollary~\ref{cor:all-curves}, and use Theorem~\ref{thm:projective-transfer}.  For a defining system $A,B,C,U,V$ giving $\eta$, keep $A,B,U$ fixed and replace $C,V$ by $kC,kV$.  Its cohomology class is $[k(AV-UC)]=k\eta$.  The proof of Theorem~\ref{thm:projective-transfer} applies to $k\gamma$ and gives zero full indeterminacy for every $k$, including $k=0$.  Since $\eta$ has exact order $n$, the values are distinct.  Theorem~\ref{thm:loss} identifies all realized triples and brackets.  This proves Theorem~\ref{thm:intro-loss-family}.
\end{proof}

The separation can also be measured in one fixed scalar target. Keeping the projection $p:Y\to C$ and the hyperplane class $h$ fixed, define
\[
 \sigma_n(a,b,kc):=(\pi_C)_*p_*(h\,k\eta)\in\C^\times.
\]
The projective-bundle formula and Theorem~\ref{thm:curve-linking} give
\begin{equation}\label{eq:scalar-fixed-family}
 \sigma_n(a,b,kc)=e_n(\alpha,\gamma)^{-k}.
\end{equation}
These values run through $\mu_n$.  No function of the one realized triple $(0,n\rho h,0)$ and its realized Massey bracket $\{0\}$ can recover these different values.  This is the precise nonfactorization assertion for ordered triples on the fixed variety $Y$, with the chosen projection and hyperplane class. It is not a numerical invariant of $Y$ alone.  The scalar formula is the classical Weil pairing; its role here is to make the failure of recovery at the secondary-realization level explicit.

\section{Scalar decomposability, divisibility, and coefficient changes}
For a smooth connected projective complex variety $X$ and $s\ge1$, define the \emph{scalar decomposable subgroup}
\[
 \Dsc^s(X):=\im\bigl(\CH^{s-1}(X)\otimes_{\mathbb Z}\mathbb C^\times
 \longrightarrow \CH^s(X,1)\bigr),
\]
where $\mathbb C^\times=\CH^1(\Spec\mathbb C,1)$ and only products with constant units are intended.

\begin{proposition}\label{prop:scalar}
The subgroup $\Dsc^s(X)$ is divisible and its morphic realization is zero.  The curve class $\theta$ lies in $\Dsc^2(C)$, the projective product value $\eta_r$ lies in $\Dsc^3(X_r)$, and the value $\eta_Z$ of Theorem~\ref{thm:ample-products} lies in $\Dsc^3(Z)$.  Consequently all these values have zero image in the respective groups $\CH^s(-,1)/m$ for every $m>0$.
\end{proposition}
\begin{proof}
Every $u\in\mathbb C^\times$ has an $m$th root, so every finite sum of scalar products admits an $m$th root inside the same image subgroup.  Thus $\Dsc^s(X)$ is divisible.  The constant unit class maps to $L^1H^1(\Spec\mathbb C)=0$, so multiplicativity kills $\Dsc^s(X)$ under morphic realization.  Formula \eqref{eq:theta} is a sum of products of point classes with units, hence $\theta\in \Dsc^2(C)$; multiplying after pullback by $h$ gives $\eta_r\in\Dsc^3(X_r)$.  The same argument with the divisor $h_2$ gives $\eta_Z\in\Dsc^3(Z)$.
\end{proof}

Scalar decomposability is not Massey indeterminacy: Theorem~\ref{thm:projective-transfer} has zero full indeterminacy even though its value is scalar decomposable.  Nor is finite order incompatible with divisibility: $\mathbb C^\times$ itself is divisible and contains all roots of unity.

For finite coefficients, Proposition~\ref{prop:uct} gives
\[
 0\longrightarrow\CH^s(X,1)/m\longrightarrow\CH^s(X,1;\Z/m)
 \longrightarrow\CH^s(X)[m]\longrightarrow0.
\]
The first term, which receives direct reduction of integral classes, need not be the whole coefficient group. Proposition~\ref{prop:scalar} shows that the specified values above have zero direct reduction. For these particular triples one can say more about the entire reduced bracket.

\begin{proposition}[Finite-coefficient vanishing for the constructed triples]\label{prop:finite-zero}
Let $\langle a,b,c\rangle_{\CH}$ be any curve based bracket of Sections~\ref{sec:curve}--\ref{sec:applications}, any member of the family in Theorem~\ref{thm:fixed-realized-triple}, or a pullback of one of these triples to a smooth complex projective variety. For every integer $m\ge2$, its bracket in the derived $\Z/m$-coefficient motivic theory is the singleton
\[
 \langle a_m,b_m,c_m\rangle_{\CH;\Z/m}=\{0\}.
\]
Its rational bracket is also $\{0\}$. Here $a_m,b_m,c_m$ denote the images under change of coefficients; no tensor-product formula for higher coefficient groups is being assumed.
\end{proposition}
\begin{proof}
Both outer classes are pulled back from $\Pic^0(C)$ for a smooth projective curve $C$. The group $\Pic^0(C)(\C)$ is divisible: it is a complex torus, and multiplication by any positive integer is surjective. Its image under pullback is therefore divisible too. By \eqref{eq:ordinary-chow-coeff},
\[
 H_M^2(X,(\Z/m)(1))=\CH^1(X,0;\Z/m)
 \cong\Pic(X)/m.
\]
Consequently $a_m=c_m=0$, for every $m$, whether or not $m$ divides the original torsion order. The bracket is thus $\langle0,b_m,0\rangle$. Zero representatives and zero null-homotopies give a zero value, and both full indeterminacy summands vanish. This proves the singleton assertion even though the intermediate coefficient groups may have nonzero $\Tor$ terms. The rational assertion follows in the same way because the original outer classes are torsion.
\end{proof}

This proposition concerns only the transported triples above, not all finite-coefficient motivic Massey products on these varieties. It is compatible with prime-local nonvanishing: reduction modulo $m$ and flat localization at a prime are different operations. In particular, exact finite order of an integral value does not imply nonvanishing of its reduced bracket.

\section{Comparison with the Lawson--Massey structure}\label{sec:lawson-comparison}
\subsection{Two distinct realization kernels}
For one fixed defined triple $a,b,c$ of codimensions $p,q,r$, put $s=p+q+r$ and define
\begin{equation}\label{eq:quotient-chain}
 Q_M:=\frac{\CH^s(X,1)}{I_M}
 \xrightarrow{\bar\rho}
 Q_{\mor}:=\frac{L^sH^{2s-1}(X)}{I_{\mor}}
 \xrightarrow{\overline{\mathrm{cl}}}
 Q_{\sing}:=\frac{H^{2s-1}(X^{\mathrm{an}};\Z)}{I_{\sing}}.
\end{equation}
The indeterminacies are \eqref{eq:intro-indet}, \eqref{eq:morphic-indet}, and \eqref{eq:singular-indet}.  Algebraically there is a two-step filtration
\[
 K_1:=\ker\bar\rho\ \subseteq\
 K_2:=\ker(\overline{\mathrm{cl}}\circ\bar\rho),
 \qquad
 K_2/K_1\cong\im\bar\rho\cap\ker\overline{\mathrm{cl}}.
\]
This is the elementary first-isomorphism theorem, not a new exactness assertion about realization on whole cohomology theories.  Moreover, not every element of these kernels must arise as a Massey value.  The examples matter because their kernel elements \emph{do} arise from defining systems and survive the relevant full indeterminacy.

The present construction realizes $K_1$.  The examples of Hu \cite{HuLawson} realize the second stage: their motivic classes have nonzero images in $Q_{\mor}$ that vanish in $Q_{\sing}$.  In this comparison the varieties and triples may differ; their quotient groups must not be identified merely because they occupy the same position in \eqref{eq:quotient-chain}.

\subsection{The Lawson--Massey product}
In \cite{HuLawson}, Theorems~1.4--1.6 give integral order-two brackets on an eightfold, an infinite family over one fixed nonzero ordered triple in singular homology, and transferred examples with ample canonical bundle.  Its Theorem~13.1 and Corollary~13.3 provide motivic lifts
\[
 \Xi_j\in\CH^5(X_L,1),\qquad X_L=S\times T\times B\times\mathbb P^1,
\]
where $S,T$ are surfaces and $B$ is an abelian threefold.  The morphic images of $\Xi_j$ survive the complete morphic indeterminacies and have zero singular-cohomology image.  Here $\Xi_j$ denotes the classes written $[\Theta_j]$ in \cite{HuLawson}; the changed letter avoids confusion with our curve cycle $\Theta$.

All curve, projective product, fixed realization, ample canonical, and formality results proved here are independent of its geometric existence theorems.  The comparison and the combined example below invoke those theorems as cited results.

\begin{proposition}[Lifting a Lawson-detected secondary value]\label{prop:lifting}
Suppose a motivic triple is defined and has a value $v\in\CH^s(X,1)$ whose image is nonzero modulo $I_{\mor}$.  Then $v\notin I_M$.  If the specified value satisfies $\mathrm{cl}\rho(v)=0$, its motivic quotient class maps to zero in the corresponding singular-cohomology quotient.
\end{proposition}
\begin{proof}
Naturality gives $\rho(I_M)\subseteq I_{\mor}$, so $v\in I_M$ would contradict the assumed nonvanishing.  The last assertion follows by composition.  The existence of a motivic defining system is essential; lifting three group elements alone does not establish it.
\end{proof}

\subsection{The complementary geometric conclusions}
\begin{table}[htbp]
\centering\small
\caption{Distinct secondary realization phenomena.  ``Zero'' for a quotient class means that its bracket contains zero.}
\label{tab:comparison}
\begin{tabularx}{\textwidth}{@{}>{\raggedright\arraybackslash}p{0.20\textwidth}>{\raggedright\arraybackslash}X>{\raggedright\arraybackslash}X@{}}
\toprule
Feature & This paper & Lawson--Massey theory \cite{HuLawson} \\
\midrule
Basic example & $C\times\mathbb P^2$, $g(C)\ge1$ & Eightfold $S\times T\times B\times\mathbb P^1$ \\
Specified values & Exact order $n$, every $n\ge2$ & Exact order two \\
Indeterminacy control & Zero on $C\times\mathbb P^r$; annihilated by a Projection formula homomorphism on products of curves & Morphic indeterminacy annihilated by a restriction homomorphism; motivic nonvanishing follows by realization \\
Morphic/Lawson bracket & Singleton $\{0\}$ & Nonzero quotient class \\
Singular realization & Singleton $\{0\}$ & Zero quotient class; a specified value has zero cycle-map image \\
Realized classes of the triple & The two outer classes vanish & All three realized classes may be nonzero \\
Realized ordered triple & $n$ distinct motivic values over one Lawson triple and bracket & Infinitely many independent Lawson values over one singular-homology triple \\
Scalar decomposability & Values belong to $\Dsc^3$ and vanish in every quotient modulo an integer & Motivic values survive modulo scalar decomposables and modulo two \\
Ample canonical bundle & Products of curves, every $d\ge2$; Lawson range for $d\ge3$ & Every $d\ge8$, with nonzero Lawson brackets \\
Formality obstruction & Motivic, over $\Z$ and every $\Z_{(\ell)}$ & Motivic and morphic, over $\Z$ and $\Z_{(2)}$ \\
\bottomrule
\end{tabularx}
\end{table}

The first construction does not supersede the second.  Its outer Chow classes are already lost under $\Pic\to\NS$; it does not establish this separation with all three morphic classes nonzero.  It detects an earlier loss of information: motivic classes are already annihilated by morphic realization.  The examples in \cite{HuLawson} detect a later loss, after the morphic/Lawson secondary class has survived.  Both use the classical secondary-operation formalism; their geometry and the realization kernels they detect are different.  Neither torsion construction supplies a nonzero rational bracket.

\begin{corollary}[Two realization stages on one variety, using \cite{HuLawson}]\label{cor:two-stage-one-variety}
Using the motivic and morphic nonvanishing results of \cite[Theorem~13.1 and Corollary~13.3]{HuLawson}, there is a smooth projective elevenfold carrying two motivic triple brackets with the following properties: the first is nonzero motivically and its entire morphic bracket is $\{0\}$; the second is nonzero modulo full morphic indeterminacy but its singular-cohomology bracket contains zero.
\end{corollary}
\begin{proof}
Take $W=X_L\times Y$ with $Y=C\times\mathbb P^2$, and pull back the two defining triples from the respective factors.  Both projections have sections, so naturality followed by restriction shows that nonvanishing modulo full indeterminacy persists.  For the first triple, both outer morphic classes are still zero, hence its entire realized bracket is $\{0\}$.  For the second, the cited result supplies a specified motivic value with nonzero morphic quotient and zero singular image; its pullback has the same properties by the section argument and naturality.  Thus the singular bracket contains zero.  We do not assert that the latter bracket has zero indeterminacy, nor that the two motivic brackets have the same degree or quotient.
\end{proof}

The corresponding motivic degrees on $W$ are $\CH^3(W,1)$ and $\CH^5(W,1)$; the morphic targets are $L^3H^5(W)$ and $L^5H^9(W)$.  In dimension eleven these are dual to $L_8H_{17}(W)$ and $L_6H_{13}(W)$, respectively.  This degree distinction prevents an unwarranted identification of the two secondary classes.

\section{Formality consequences}\label{sec:formality}
\begin{definition}
An associative weight-graded differential graded $R$-algebra $A$ is \emph{formal over $R$} if a zigzag of weight-preserving multiplicative quasi-isomorphisms connects it to its cohomology algebra $H^*(A)$ with zero differential.
\end{definition}

\begin{lemma}[Classical Massey obstruction and localization]\label{lem:formality}
Every defined triple bracket in a formal algebra contains zero.  A nonzero integral triple quotient class of exact order $\ell$, with $\ell$ prime, remains nonzero after derived extension to $\Z_{(\ell)}$.
\end{lemma}
\begin{proof}
The first assertion is classical \cite{May1969}: Proposition~\ref{prop:naturality} identifies full triple brackets under multiplicative quasi-isomorphisms, while in an algebra with zero differential one may choose zero null-homotopies.  For the second assertion, flat localization commutes with cohomology, with the images of multiplication defining the two indeterminacy summands, and with quotienting by their sum.  A nonzero element of order $\ell$ cannot be killed by an integer prime to $\ell$.
\end{proof}

\begin{corollary}[Motivic nonformality]\label{cor:nonformal-genus}
Let $C$ be a smooth connected complex projective curve of positive genus and let $r\ge0$. The associative weight-graded models of the multiplicative motivic theory of $C\times\mathbb P^r$ are nonformal over $\Z$ and over every $\Z_{(\ell)}$.  The same assertion holds for every product $Z=C_1\times\cdots\times C_d$ in Theorem~\ref{thm:ample-products}; in particular, it holds for the ample canonical examples of Theorem~\ref{thm:intro-ample}.
\end{corollary}
\begin{proof}
For every prime $\ell$, Corollary~\ref{cor:all-curves}, Theorem~\ref{thm:projective-transfer}, or Theorem~\ref{thm:ample-products}, as appropriate, gives a motivic Massey class of exact order $\ell$ modulo its full indeterminacy.  Apply Lemma~\ref{lem:formality}.
\end{proof}

\begin{corollary}[Formality obstruction after transfer]\label{cor:transferred-nonformal}
Let $V$ be a variety in Corollary~\ref{cor:nonformal-genus}. If $f:W\to V$ is a surjective generically finite morphism of degree $e$ between smooth connected projective varieties of the same dimension, then the multiplicative motivic theory of $W$ is nonformal over $\Z$ and over $\Z_{(\ell)}$ for every prime $\ell\nmid e$. Projective-bundle pullbacks and birational morphisms between smooth projective varieties preserve nonformality over every $\Z_{(\ell)}$. These statements apply in particular to the cyclic covers of Corollary~\ref{cor:cyclic-cover}.
\end{corollary}
\begin{proof}
For each prime $\ell\nmid e$, choose on $V$ the bracket with quotient class of exact order $\ell$. Proposition~\ref{prop:coprime-transfer} preserves that order; Lemma~\ref{lem:formality} then obstructs integral and $\Z_{(\ell)}$-formality. There is always a prime not dividing $e$, which gives the integral assertion. For projective bundles and birational morphisms use Corollary~\ref{cor:bundle-birational}, whose trace scalar is one.
\end{proof}

These are obstructions for the motivic multiplicative enhancement.  They do not show that the morphic theory or integral singular cochains are nonformal: the particular realized brackets constructed here are zero.  They also do not give rational nonformality.  After rationalization the two outer torsion classes and the specified torsion value vanish.  Prime localization and rationalization have different effects and must not be conflated.

\section{Scope and significance}
The secondary invariant \eqref{eq:secondary-invariant} is obtained from the classical motivic Massey operation; its geometric torsion linking interpretation has the precedents discussed in the introduction.  The substantive separation proved here is Theorem~\ref{thm:fixed-realized-triple}: different full indeterminacy Chow invariants, even in one fixed target group, can have the same morphic classes and the same zero Lawson--Massey bracket.  This is more precise than exhibiting an unspecified element of a realization kernel.

The curve construction gives exact torsion by an actual cubical boundary and proves nonvanishing by pushforward to the Weil pairing.  Projective products preserve zero full indeterminacy.  Products of curves need a different argument: the full indeterminacy is not computed as zero, but the Projection formula homomorphism annihilates it and proves exact order in the quotient.  The coprime degree transfer preserves the exact order of the secondary quotient, without identifying its new indeterminacy with the old one. Proposition~\ref{prop:finite-zero} strengthens direct reduction of the specified values to vanishing of the complete reduced brackets. The comparison with \cite{HuLawson} concerns a different realization step and is not used in the curve based proofs.

Several limitations are essential.  All specified nonzero values constructed here are integral torsion. A bracket with nonzero indeterminacy may contain other values; no torsion assertion is made about every element of that coset.  Their outer realized classes are zero; the result is not a construction with three nonzero Lawson classes.  The values are scalar decomposable, so their images in the scalar-indecomposable quotient and in every ordinary quotient modulo an integer are zero.  None of these facts removes their nonzero full Massey quotient.  No rational nonvanishing, general operation on ordinary Chow groups of arbitrary singular varieties, or vanishing of Deligne regulators is asserted.  Finally, ``beyond Lawson--Massey'' is the nonfactorization statement for the realized ordered triple and its bracket, not a claim that a fixed variety is distinguished by all possible cohomological or homotopical invariants.

\end{document}